\documentclass{amsart}
\usepackage{amssymb}
\usepackage[utf8]{inputenc}
\usepackage[T1]{fontenc}
\usepackage[english]{babel}
\usepackage{lmodern,microtype}
\usepackage{mathrsfs,enumitem}
\usepackage{tikz}
\usetikzlibrary{decorations,decorations.pathreplacing}
\usepackage{mathtools}
\usepackage{hyperref}

\setenumerate{label=\upshape(\arabic*)}

\newtheorem{theorem}{Theorem}
\newtheorem{claim}[theorem]{Assertion}

\renewcommand{\le}{\leqslant}
\renewcommand{\ge}{\geqslant}

\renewcommand{\geq}{\geqslant}
\newcommand{\abs}[1]{\left\lvert#1\right\rvert}

\newcommand{\bigo}[1]{O\mathopen{}\left(#1\right)}

\newtheorem{theo2}{Theorem}
\begin{document}

\author{Stéphane Bessy} \author{Daniel Gon\c{c}alves}
\address[S. Bessy, D. Gon\c{c}alves]{AlGCo, LIRMM,
  Universit\'e Montpellier 2, France}
\email{\{goncalves,bessy\}@lirmm.fr}

\author{Jean-Sébastien Sereni}
\address[J.-S. Sereni]{Centre National de la Recherche Scientifique (LORIA),
  Vand\oe uvre-lès-Nancy, France.}
\thanks{The second and third author's work were partially supported by the
  French \emph{Agence Nationale de la Recherche} under references
  \textsc{ANR-12-JS02-002-01} and \textsc{anr 10 jcjc 0204 01}, respectively.}
\email{sereni@kam.mff.cuni.cz}

\title{Two floor building needing eight colors}

\date{}

\begin{abstract}
  Motivated by frequency assignment in office blocks, we study the chromatic
  number of the adjacency graph of $3$-dimensional parallelepiped
  arrangements. In the case each parallelepiped is within one floor, a direct
  application of the Four-Colour Theorem yields that the adjacency graph has
  chromatic number at most $8$. We provide an example of such an arrangement
  needing exactly $8$ colours.  We also discuss bounds on the chromatic number
  of the adjacency graph of general arrangements of $3$-dimensional
  parallelepipeds according to geometrical measures of the parallelepipeds
  (side length, total surface or volume).
\end{abstract}
\maketitle

\section{Introduction}
The Graph Colouring Problem for Office Blocks was raised by BAE Systems at the
53rd European Study Group with Industry in 2005~\cite{AGG+05}.  Consider an
office complex with space rented by several independent organisations.  It is
likely that each organisation uses its own wireless network (WLAN) and ask for
a safe utilisation of it.  A practical way to deal with this issue is to use a
so-called ``stealthy wallpaper'' in the walls and ceilings shared between
different organisations, which would attenuate the relevant frequencies. Yet,
the degree of screening produced will not be sufficient if two distinct
organisations have adjacent offices, that is, two offices in face-to-face
contact on opposite sides of just one wall or floor-ceiling. In this case, the
WLANs of the two organisations need to be using two different channels (the
reader is referred to the report by Allwright \emph{et al.}~\cite{AGG+05} for
the precise technical motivations).

This problem can be modeled as a graph coloring problem by building a
\emph{conflict graph} corresponding to the office complex: to each
organisation corresponds a vertex, and two vertices are adjacent if the
corresponding territories share a wall, floor, or ceiling area. The goal is to
assign a color (frequency) to each vertex (organisation) such that adjacent
vertices are assigned distinct colors. In addition, not every graph may occur
as the conflict graph of an existing office complex. However, the structure of
such conflict graphs is not clear and various fundamental questions related to
the problem at hands were asked.  Arguably, one of the most natural questions
concerns the existence of bounds on the chromatic number of such conflict
graphs. More specifically, which additional constraints one should add to the
model to ensure ``good'' upper bounds on the chromatic number of conflict
graphs? These additional constraints should be meaningful regarding the
practical problem, reflecting as much as possible real-world situations.
Indeed, as noted by Tietze~\cite{Tie05}, complete graphs of arbitrary size are
conflict graphs, that is, for every integer $n$, there can be $n$
organisations whose territories all are in face-to-face contact with each
other. The reader is referred to the paper by Reed and Allwright~\cite{ReAl08}
for a description of Tietze's construction.  Besicovitch~\cite{Bes47} and
Tietze~\cite{Tie65} proved that this is still the case if the territories are
asked to be convex polyhedra.

An interesting condition is when the territories are required to be
\emph{rectangular parallelepipeds} (sometimes called \emph{cuboids}), that is,
a $3$-dimensional solid figure bounded by six rectangles aligned with a fixed
set of Cartesian axes. For convenience, we shall call \emph{box} a rectangular
parallelepiped. When all territories are boxes, the \emph{clique number} of
any conflict graph, that is, the maximum size of a complete subgraph, is at
most $4$.  However, Reed and Allwright~\cite{ReAl08} and also later Magnant
and Martin~\cite{MaMa11} designed arrangements of boxes that yield conflict
graphs requiring an arbitrarily high number of colours.

\noindent On the other hand, if the building is assumed to have floors (in the
usual way) and each box is \emph{$1$-floor}, i.e. restricted to be within one
floor, then the chromatic number is bounded by $8$: on each floor, the
obtained conflict graph is planar and hence can be coloured using $4$
colours~\cite{ApHa77,AHK77}.  It is natural to ask whether this bound is
tight. As noted during working sessions in Oxford (see the acknowledgments),
it can be shown that up to $6$ colours can be needed, by using an arrangement
of boxes spanning three floors. Such a construction is shown in
Figure~\ref{fig:6}.

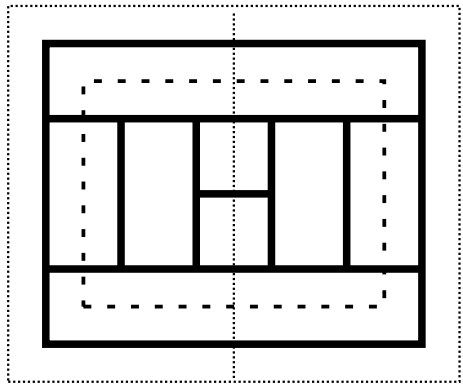
\begin{figure}[!ht]
    \begin{center}
\begin{tikzpicture}[line width = 1mm]
\tikzset{BarreStyle/.style = {opacity=.4,line width=4 mm,line cap=round,color=#1}}
    \draw (0,0) -- (0,4) -- (5,4) -- (5,0) -- cycle ;
    \foreach \i in {1,2,3,4} { \draw (\i,1) -- (\i,3); }
    \draw (2,2)--(3,2);
    \draw (0,1)--(5,1);
    \draw (0,3)--(5,3);

\draw [line width=.3mm,densely dotted] (-.5,-.5)--(-.5,4.5)--(5.5,4.5)--(5.5,-.5)--cycle;
\draw [line width=.3mm,densely dotted] (2.5,4.4)--(2.5,-.45);
\draw [line width=.5mm, loosely dashed] (.5,.5)--(.5,3.5)--(4.5,3.5)--(4.5,.5)--cycle;
\end{tikzpicture}
\end{center}
\caption{An arrangement of $1$-floor boxes spanning three floors and requiring
six colours. The solid, dotted, and dashed lines indicate the middle, top, and
bottom floors, respectively.}\label{fig:6}
\end{figure}

The purpose of this note is to show that the upper bound is actually tight.
More precisely, we shall build an arrangement of $1$-floor boxes that spans
two floors and yields a conflict graph requiring $8$ colours. From now on, we
shall identify a box arrangement with its conflict graph for convenience. In
particular, we assign colors directly to the boxes and define the
\emph{chromatic number of an arrangement} as that of the associated conflict
graph.
\begin{theorem}
\label{thm:8}
  There exists an arrangement of $1$-floor boxes spanning two floors with
  chromatic number $8$.
\end{theorem}

The boxes considered in Theorem~\ref{thm:8} have one of their geometrical
measures bounded: their height is at most one floor.  We also discuss bounds
on the chromatic number of box arrangements with respect to some other
geometrical measures: the side lengths, the surface area and the volume. More
precisely, assuming that boxes have integer coordinates, we obtain the
following.
\begin{theorem}
 \label{thm:geom}
We consider a box arrangement $A$ with integer coordinates.
\begin{enumerate}
   \item If there exists one fixed dimension such that every box in
      $A$ has length at most $\ell$ in this dimension, then
      $A$ has chromatic number at most $4(\ell+1)$.
   \item If for each box, there is one dimension such that the
      length of this box in this dimension if at most $\ell$, then $A$
      has chromatic number at most $12(\ell+1)$.
   \item If the total surface area of each box in $A$ is
      at most $s$, then $A$ has chromatic number at most
      $9\sqrt[3]{4s}+13$.
   \item If the volume of each box in $A$ is at most $v$,
      then $A$ has chromatic number at most $24\sqrt[4]{6v}+13$.
\end{enumerate}
\end{theorem}
\noindent In the next section, we give the proof of Theorem~\ref{thm:8} and in
the last section we indicate how to obtain the bounds given in
Theorem~\ref{thm:geom}.

\section{Proof of Theorem~\ref{thm:8}}
We shall construct an arrangement of $1$-floor boxes that is not
$7$-colorable.  Before that, we need the following definition.  Consider an
arrangement $A$ and let $A_1$, $A_2$ and $A_3$ be (non-necessarily disjoint)
subsets of the boxes in $A$. Given a proper coloring $c$ of $A$, let $C_i$ be
the set of colors used for the boxes in $A_i$, for $i\in\{1,2,3\}$.  The
\emph{signature $\sigma_A(c)$ of $A$ with respect to $c$} is defined to be
$a_xb_yc$, where $a\coloneqq\abs{C_1}$, $x\coloneqq\abs{C_1 \cup C_2}$,
$b\coloneqq\abs{C_2}$, $y\coloneqq\abs{C_2 \cup C_3}$, and
$c\coloneqq\abs{C_3}$.  The collection of all signatures can be endowed with a
partial order: if $s=a_xb_yc$ and $s'=a'_{x'}b'_{y'}c'$ are two signatures,
then $s\le s'$ if $a\le a'$, $x\le x'$, $b\le b'$, $y\le y'$ and $c\le c'$.

To build the desired arrangement, we use the arrangement $X$ of $1$-floor
boxes described in Figure~\ref{fig-X} as a building brick. The arrangement $X$
has three specific \emph{regions}, $X_1$, $X_2$ and $X_3$. We also abusively
write $X_1$, $X_2$ and $X_3$ to mean the subsets of boxes of $X$ respectively
intersecting the regions $X_1$, $X_2$ and $X_3$ (note that some boxes may
belong to several subsets). We start by giving some properties of the
signatures of $X$ with respect to proper colorings and according to the three
subsets  $X_1$, $X_2$ and $X_3$.

\begin{figure}[!ht]
\begin{center}
\begin{tikzpicture}[line width=1mm,rotate=90]
\tikzset{BarreStyle/.style = {opacity=.3,line width=6 mm,color=#1}}
    \draw (0,0) -- (3,0) -- (3,5) -- (0,5) -- cycle ;
    \draw (0,1) -- (3,1) ;
    \draw (0,2) -- (3,2) ;
    \draw (2,3) -- (3,3) ;
    \draw (0,4) -- (3,4) ;
    \draw (0,5) -- (3,5) ;
    \draw (1,2) -- (1,4) ;
    \draw (2,2) -- (2,4) ;
\draw [BarreStyle=red] (.5,4.8) to (.5,.2) ;
\draw [BarreStyle=blue] (1.5,4.8) to (1.5,.2) ;
\draw [BarreStyle=green] (2.5,4.8) to (2.5,.2) ;

\draw node at (.5,-.4) {$X_1$};
\draw node at (1.5,-.4) {$X_2$};
\draw node at (2.5,-.4) {$X_3$};
\end{tikzpicture}
\end{center}
\caption{The gadget $X$ with the regions $X_1$, $X_2$ and $X_3$.}
\label{fig-X}
\end{figure}
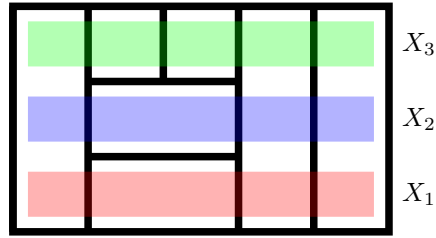

\begin{claim}\label{claim-X}
  For every proper coloring $c$ of $X$,
  \begin{enumerate}
     \item $\sigma_X(c) \geq 3_32_44$,\label{X33244}
     \item $\sigma_X(c) \geq 3_43_32$, or\label{X34332}
     \item $\sigma_X(c) \geq 2_33_43$.\label{X23343}
  \end{enumerate}
\end{claim}
The proof of this assertion does not need any insight, we thus omit it.
However, the reader interested in checking its accuracy should first note that
one can restrict to the cases where the three vertical (in Figure~\ref{fig-X})
boxes are respectively colored either $1$, $2$ and $1$; or $1$, $1$ and $2$;
or $1$, $2$ and $3$.

\begin{figure}[!ht]
\begin{center}
\begin{tikzpicture}[line width = 1mm,scale=.78]
\tikzset{BarreStyle/.style = {opacity=.3,line width=6 mm,color=#1}}

\draw (0,0)--(15,0)--(15,5)--(0,5)--cycle;

\foreach \i in {1,3,4,5,6,8,9,10,11,13,14} { \draw (\i,0)--(\i,5) ; }

\draw (1,2)--(3,2);
\draw (1,3)--(3,3);
\draw (2,3)--(2,5);

\draw (6,1)--(8,1);
\draw (6,3)--(8,3);
\draw (7,3)--(7,5);

\draw (11,4)--(13,4);
\draw (11,2)--(13,2);
\draw (12,0)--(12,2);

\draw [BarreStyle=red] (0.2,1.5) to (14.8,1.5) ;
\draw [BarreStyle=blue] (0.2,2.5) to (14.8,2.5) ;
\draw [BarreStyle=green] (0.2,3.5) to (14.8,3.5) ;

\draw node at (-.5,1.5) {$Y_1$};
\draw node at (-.5,2.5) {$Y_2$};
\draw node at (-.5,3.5) {$Y_3$};

\draw [decorate,decoration={brace,amplitude=5pt},line width=.3mm] (0.1,5.2) -- (4.9,5.2) node[midway,above=.4em] {$X^1$};
\draw [decorate,decoration={brace,amplitude=5pt},line width=.3mm] (5.1,5.2) -- (9.9,5.2) node[midway,above=.4em] {$X^2$};
\draw [decorate,decoration={brace,amplitude=5pt},line width=.3mm] (10.1,5.2) -- (14.9,5.2) node[midway,above=.4em] {$X^3$};
\end{tikzpicture}
\end{center}
\caption{The gadget $Y$ with the regions $Y_1$, $Y_2$ and $Y_3$.}
\label{fig-Y}
\end{figure}
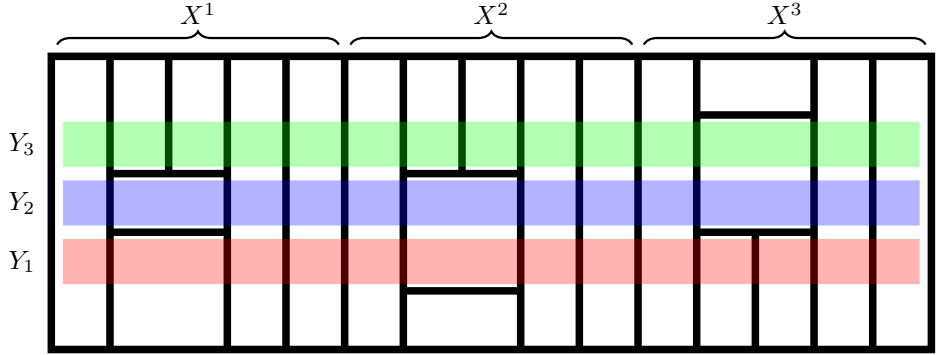

Now, the arrangement $Y$ is obtained from three copies $X^1$, $X^2$ and $X^3$
of the arrangement $X$. We define three regions $Y_1$, $Y_2$ and $Y_3$ on $Y$
as depicted in Figure~\ref{fig-Y}.  As previously, we also write $Y_1$, $Y_2$
and $Y_3$ for the subsets of boxes intersecting the region $Y_1$, $Y_2$ and
$Y_3$, respectively. We set $X_i^j\coloneqq Y_i\cap X^j$ for
$(i,j)\in\{1,2,3\}^2$.

\begin{claim}
\label{claim-Y}
In any proper coloring of $Y$, at least four colors are used in one
of the three regions.
\end{claim}
\begin{proof}
Suppose on the contrary that there is a proper coloring $c$ of $Y$ with at
most three colors in each of $Y_1$, $Y_2$ and $Y_3$. For $i\in\{1,2,3\}$, the
restriction of $c$ to $X^i$ is a proper coloring of $X^i$, which we identify
to $c$.  The condition on $c$ implies that none of $\sigma_{X^1}(c)$,
$\sigma_{X^2}(c)$ and $\sigma_{X^3}(c)$ fulfills inequality~\ref{X33244} of
Assertion~\ref{claim-X}. In particular, note that exactly $3$ different
colours appear on $X_2^1$, and they also appear on $X_2^2$ and on $X_2^3$.
Since $X_1^2=X_2^2$, these three colours appear on $Y_1$.  Similarly, since
$X_3^3=X_2^3$, these three colours appear on $Y_3$.

Assume now that $\sigma_{X^1}(c)$ satisfies~\ref{X34332}.  Then exactly three
colours appear on $X_1^1$, one of which does not appear on $X_2^1$ as
$\abs{c(X_1^1\cup X_2^1)}\ge4$. Thus in total at least four colours appear on
$X_1^1\cup X_1^2\subset Y_1$, which contradicts our assumption on $c$.  It
remains to deal with the case where $\sigma_{X^1}(c)$ fulfills~\ref{X23343} of
Assertion~\ref{claim-X}. Thanks to the symmetry of~\ref{X34332}
and~\ref{X23343} with respect to the regions $X_1$ and $X_3$, the same
reasoning as above applied to $X_3^3$ instead of $X_1^2$ yields that four
colours appear on $Y_3$, a contradiction.
\end{proof}

To finish the construction, we need the following definition. Consider
two copies $Y^1$ and $Y^2$ of $Y$. The regions $Y_i^1$ and $Y_j^2$
\emph{fully overlap} if every box in $Y_i^1$ is in face-to-face
contact with every box in $Y_j^2$. Observe that for every pair
$(i,j)\in\{1,2,3\}$, there exists a $2$-floor arrangement of $Y^1$ and
$Y^2$ such that $Y_i^1$ and $Y_j^2$ fully overlap: it is obtained by
rotating $Y^2$ ninety degrees, adequately scaling it (i.e. stretching it
horizontally) and placing it on top of $Y^1$.

We are now in a position to build the desired arrangement $Z$ spanning
two floors. To this end, we use several copies of $Y$.  The first
floor of $Z$ is composed of seven parallel copies $Y^1,\ldots,Y^7$ of
$Y$ (drawn horizontally in Figure~\ref{fig-Z}).  The second floor of
$Z$ is composed of fifteen parallel copies of $Y$ (drawn vertically in
Figure~\ref{fig-Z}): for each $j\in\{1,2,3\}$ and each
$i\in\{2,\ldots,6\}$, a copy $Y(i,j)$ of $Y$ is placed such that the
first region of $Y(i,j)$ fully overlap the regions
$Y_j^1,\ldots,Y_{j}^{i-1}$, the second region of $Y(i,j)$ fully
overlaps the region $Y_j^i$, and the third region of $Y(i,j)$ fully
overlaps the regions $Y_j^{i+1},\ldots,Y_{j}^{7}$.

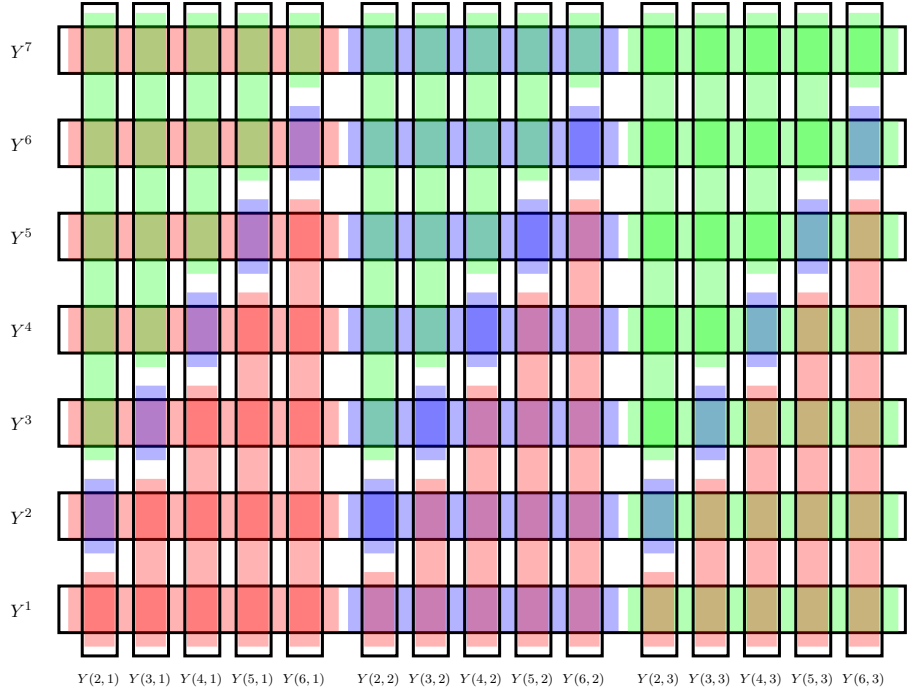
\begin{figure}
\begin{tikzpicture}[line width = .4mm,scale=.62,every node/.style={scale=.62}]
\tikzset{BarreStyle/.style = {opacity=.3,line width= 6 mm,color=#1}}
  \foreach \i in {1,...,7} {

  \draw (0,2*\i)--(18.15,2*\i)--(18.15,-1+2*\i)--(0,-1+2*\i)--cycle ;
  \draw node at (-.8,-.5+2*\i) {\large $Y^\i$} ;
  \draw [BarreStyle=red] (.2,1.5+2*\i-2) to (6,1.5+2*\i-2) ;
  \draw [BarreStyle=blue] (6.2,1.5+2*\i-2) to (12,1.5+2*\i-2) ;
  \draw [BarreStyle=green] (12.2,1.5+2*\i-2) to (18,1.5+2*\i-2) ;
}

\tikzset{BarreStyle/.style = {opacity=.3,line width=4 mm,color=#1}}
\foreach \i in {0,1,2} {
  \foreach \j in {0,1,2,3,4} {
    \draw (.5+1.1*\j+6*\i,.5)--(.5+1.1*\j+6*\i,14.5)--(1.25+6*\i+1.1*\j,14.5)--(1.25+6*\i+1.1*\j,.5)--cycle ;
    \draw [BarreStyle=red] (.87+1.1*\j+6*\i,2.3+2*\j) to (.87+1.1*\j+6*\i,0.7) ;
    \draw [BarreStyle=blue] (.87+1.1*\j+6*\i,2.7+2*\j) to (.87+1.1*\j+6*\i,4.3+2*\j) ;
    \draw [BarreStyle=green] (.87+1.1*\j+6*\i,14.3) to (.87+1.1*\j+6*\i,4.7+2*\j) ;

  }
}
\foreach \i in {1,2,3} {
  \foreach \j in {2,...,6} {
  \draw node at (\j*1.1-2.2+.3+.55+\i*6-6,0) {\footnotesize $Y(\j,\i)$} ;
  }
}
\end{tikzpicture}
\caption{Schematic view of the arrangement $Z$.}
\label{fig-Z}
\end{figure}

Consider a proper coloring of $Z$. Assertion~\ref{claim-Y} ensures that each
copy of $Y$ in $Z$ has a region for which at least four different colours are
used.  In particular, there exists $j\in\{1,2,3\}$ such that three regions
among $Y_j^1,\ldots,Y_j^7$ are colored using four colours. Let these regions
be $Y_j^{i_1}$, $Y_j^{i_2}$ and $Y_{j}^{i_3}$ with $1\le i_1<i_2<i_3\le7$.
Now, consider the arrangement $Y(i_2,j)$. By Assertion~\ref{claim-Y}, there
exists $k\in\{1,2,3\}$ such that the $k$-th region of $Y(i_2,j)$ is also
colored using at least four different colors.  Consequently, as this region
and the region $Y_j^{i_k}$ fully overlap, they are colored using at least
eight different colors. This concludes the proof.

\section{Bounds with respect to geometrical measures}
In this part, we provide bounds on the chromatic number of boxes arrangements
provided that the boxes satisfy some geometrical constraints. Namely, we prove
Theorem~\ref{thm:geom}, which is recalled here for the reader's ease.
\begin{theo2}
We consider a box arrangements $A$ with integer
coordinates.
\begin{enumerate}
   \item If there exists one fixed dimension such that every box in $A$ has
      length at most $\ell$ in this dimension, then $A$ has chromatic number
      at most $4(\ell+1)$.\label{toutes}
   \item If for each box, there is one dimension such the length of this box
      in this dimension if at most $\ell$, then $A$ has chromatic number at
      most $12(\ell+1)$.\label{une}
   \item If the total surface area of each box in $A$ is at most $s$, then $A$
      has chromatic number at most $9\sqrt[3]{4s}+13$.\label{surface}
   \item If the volume of each box in $A$ is at most $v$, then $A$ has
      chromatic number at most $24\sqrt[4]{6v}+13$.\label{volume}
\end{enumerate}
\end{theo2}
\begin{proof}\mbox{}

\noindent
\ref{toutes} The conflict graph corresponding to an arrangement where the
boxes have height at most $\ell$ can be vertex partitioned into $\ell+1$
planar graphs $P_0,\ldots,P_\ell$.  Indeed if the distance between the levels
of two boxes is at least $\ell+1$, then these two boxes are not adjacent. So
the planar graphs are obtained by assigning, for each $x$, all the boxes that
have their floor at level $x$ to be in the graph $P_k$ where $k\coloneqq x
\mod (\ell+1)$. Consequently, the whole conflict graph has chromatic number at
most $4(\ell+1)$.

\smallskip
\noindent\ref{une} The boxes can be partitioned into three sets according to
the dimension in which the length is bounded. In other words, $A$ is
partitioned into $U_1$, $U_2$ and $U_3$ such that for each $i\in\{1,2,3\}$,
all boxes in $U_i$ have length at most $\ell$ in dimension $i$.
Consequently,~\ref{toutes} ensures that each of $U_1$, $U_2$ and $U_3$ has
chromatic number at most $4(\ell+1)$ and, therefore, $A$ has chromatic number
at most $3\cdot 4(\ell+1)=12(\ell+1)$.

\smallskip
\noindent\ref{surface} For each box, the minimum length taken over all three
dimensions is at most $\sqrt{s}$, and thus~\ref{une} implies that the
chromatic number of $A$ is $\bigo{\sqrt{s}}$.  However, one can be more
careful. Let us fix a positive integer $\ell$, to be made precise later.  The
set of boxes is partitioned as follows. Let $U$ be the set of boxes with
lengths in every dimension at least $\ell$ and let $R$ be the set all
remaining boxes, that is, $R\coloneqq A\setminus U$.  By~\ref{une}, the
arrangement $R$ has chromatic number at most $12\ell$.  Now consider a box $B$
in $U$ with dimensions $x$, $y$ and $z$, each being at least $\ell$. We shall
give an upper bound on the number of boxes of $U$ that can be adjacent to $B$.
The surface of a face of a box in $U$ is at least $\ell^2$.  So in $U$ there
are at most $s/\ell^2$ that have a face totally adjacent to a face of $B$.
Some boxes of $U$ could also be adjacent to $B$ without having a face totally
adjacent to a face of $B$. In this case, such a box is adjacent to an edge of
$B$. For an edge of length $w$, there are at most $w/\ell+1$ such boxes. So
the number of boxes of $U$ adjacent to $B$ but having no face totally adjacent
to a face of $B$ is at most $4(x+y+z)/\ell+12$. Since $\ell\le\min\{x,y,z\}$,
we deduce that $2\ell(x+y+z)\le 2xy+2yz+2xz\le s$.  Hence the total number of
boxes in $U$ that are adjacent to $B$ is at most
$s/\ell^2+2s/\ell^2+12=3s/\ell^2+12$.  Consequently, by degeneracy, $U$ has
chromatic number at most $3s/\ell^2+13$. Therefore, $A$ has chromatic number
at most $12\ell+3s/\ell^2+13$.  Setting $\ell\coloneqq\sqrt[3]{s/2}$ yields
the upper bound $9\sqrt[3]{4.s}+13$.

\smallskip
\noindent\ref{volume} Once again, for a fixed parameter $\ell$ to be made
precise later, the set of boxes is partitioned into two parts: the part $U$,
composed of all the boxes with lengths in every dimension at least $\ell$ and
the part $R$, composed of all the remaining boxes.  By~\ref{une}, we know that
$R$ has chromatic number at most $12\ell$.  Let $B$ be a box in $U$ with
dimensions $x$, $y$ and $z$.  Since $\ell\le \min\{x,y,z\}$, the volume $v_B$
of $B$ satisfies that $6v\ge6v_B=6xyz\ge2(\ell xy+\ell xz+\ell yz)=\ell s_B$,
where $s_B$ is the total surface area of $B$.  So every box in $U$ has total
surface area at most $6v/\ell$ and thus~\ref{surface} implies that $U$ has
chromatic number at most $9\sqrt[3]{4.6v/\ell}+13$.  Therefore, $A$ has
chromatic number at most $9\sqrt[3]{24v/\ell}+12\ell+13$. Setting $\ell$ to be
$\sqrt[4]{3v/8}$ yields the upper bound $24\sqrt[4]{6v}+13$.
\end{proof}

In the previous theorem, we are mainly concerned with the order of magnitude
of the functions of the different parameters. However, even in this context,
we do not have any non trivial lower bound on the corresponding chromatic
numbers.

\section*{Acknowledgments.} The third author thanks Louigi Addario-Berry,
Frédéric Havet, Ross Kang, Colin McDiarmid and Tobias Müller for stimulating
discussions on the topic of this paper during a stay in Oxford, in 2005.

\end{document}